\documentclass[12pt,reqno]{amsart}
 \usepackage{graphicx}
\usepackage{amssymb, amsmath}
\usepackage{mathrsfs}
 \usepackage{subfig}

\newcommand {\PP}{{I\kern-.3em P}}

\newcommand {\HH}{{\mathbb{H}}}
\newcommand {\RR}{{\mathbb{R}}}
\newcommand {\NN}{{\mathbb {N}}}

\newcommand{\beq}{\begin{equation}}
\newcommand{\eeq}{\end{equation}}
\newcommand{\beqq}{\begin{equation*}}
\newcommand{\eeqq}{\end{equation*}}

\newcommand{\pmtx}[1]{\begin{pmatrix}#1\end{pmatrix}}

 \numberwithin{equation}{section}

{\begin{enumerate}[\bfseries 1.]}{\end{enumerate}}
{\begin{enumerate}[\bfseries a.]}{\end{enumerate}}

\def\d{\displaystyle}

\def\p{$p\,\,$}

\def\btheta{\bar{\theta}}
\def\bx{\bar{x}}
\def\by{\bar{y}}
\def\bT{{\bf T}}
\def\bB{{\bf B}}
\def\bF{{\bf F}}
\def\bD{{\bf D}}
\def\Oa{\Omega_{\alpha}}
\newtheorem{thm}{Theorem}
\newtheorem{lemma}{Lemma}
\newtheorem{prop}{Proposition}

\begin{document}

\title{Sequences and dynamical systems associated with canonical approximation by rationals}
\author{Andrew Haas }
 \email{haas@math.uconn.edu}
 \address{University of Connecticut, Department of Mathematics,
 Storrs, CT 06269}
\begin{abstract}
We study metrical properties of various subsequences associated to the sequence of rational approximants coming from the continued fraction of an irrational number. Our methods  build upon Bosma, Jager and Wiedijk's proof of the Doblin-Lenstra conjecture  as well as Jager's subsequent treatment of the sequence of approximation pairs. 
 
 \end{abstract}



 
\maketitle
    \markboth{Sequences associated with approximation by rationals}
 {Andrew Haas}

\section{introduction}
Each irrational $x\in(0,1)$ has a unique representation as an infinite regular continued fraction $[a_1,a_2,\ldots]$. By virtue of this representation, it is possible to associate to $x$ the sequence of convergents $p_n/q_n=[a_1,\ldots ,a_n]$,\cite{HW}.  Define $\theta_n(x)= q_n|q_nx-p_n|$. We are interested in studying generic properties of certain subsequences of the sequences $\{p_n/q_n\}$,  $\{\theta_n\}$ and $\{(\theta_n,\theta_{n+1})\}$. Our methods make use of ergodic theory and hyperbolic geometry and build upon Bosma, Jager and Wiedijk's proof of the Doblin-Lenstra conjecture (refered to as D-L) \cite{bjw,knuth,dankraa} as well as Jager's subsequent treatment of the sequence of approximation pairs \cite{jager, dankraa}. 

For $\alpha\in(0,1]$, define the subsequence $ \bar{\Theta}(\alpha)=\{ \theta_{n_k}\}$, where $\theta_{n_k}<\alpha$ and if  $\theta_n<\alpha$ then  $\theta_n=\theta_{n_k} $ for some $k$. There is of course the associated subsequence  of   convergents $\{p_{n_k}/q_{n_k}\}$. For simplicity we suppress the reference to both $x$ and $\alpha$ and denote these two subsequences by $\bar{\Theta}(\alpha)=\{\bar{\theta}_k\}$ and $\{\bar{p}_k/\bar{q}_k\}.$  One interesting fact that is evident in our treatment is  that, while for $\alpha\leq 1/2$  the sequence $\bar{\Theta}(\alpha)$ is equidistributed in the interval $(0,\alpha)$, the corresponding sequence of pairs is far from  evenly distributed in its natural domain. For $\alpha\geq 1/2$ we  derive the actual density function for the distribution. We also see how D-L and a famous theorem of L\'evy on the growth of the numerator of convergents can be realized in this setting.

As in \cite{bjw} and \cite{jager}, the main tool is the natural automorphic extension $\bT$ of the Gauss map. The subsequences are associated to a family of automorphisms defined by taking the first return to a region $\Oa$. These first return maps are extremely interesting in their own right. While $\bT$ is Bernoulli and therefore qualifies as being chaotic, with decreasing $\alpha$ the first return maps exhibit increasing complexity and appear to do so in a manner reminiscent of the structure of the Markoff spectrum \cite{cf}. Techniques form hyperbolic geometry will be employed to gain some understanding of the structure of the first return maps and to point the way for dealing with cases when $\alpha<1/2$.

\section{basic properties of the sequences and the natural extension} 
 \subsection{The natural extension of the Gauss map}
 
We shall   begin by discussing the Gauss map, its natural automorphic extension $\bT$ and a related group of first return maps. These maps are used to pinpoint the sequence of values $n_k$ mentioned above as well as providing the framework for analyzing dynamical properties of the sequence of thetas and the pairs.

The classical Gauss map is defined  on the open unit interval $I=(0,1)$ by $T(x)=1/x-[1/x]$, where we use $[r]$ to denote the greatest integer less than or equal to $r$. The map has several nice properties. First $T$ acts as a shift on the continued fraction expansions: $T([a_1,a_2,\ldots])=[a_2,a_3,\ldots]$.  Secondly, $T$ is ergodic with respect to Lebesgue measure and has the absolutely continuous invariant probability measure  $(\log 2(1+x))^{-1}dx$,  \cite{dankraa}. 

  There is a particularly simple and useful realization of the natural automorphic extension of $\bT$  \cite{keller} due to Nakada, et.al. \cite{nakada, dankraa}. 
See also \cite{af0}. We shall use a closely related realization of $\bf{T}$, defined on  $\Omega=I\times (-\infty,-1)$ by
\beqq
{\bf T}(x,y)=(1/x-[1/x],1/y-[1/x]),\hskip .8in
\eeqq
  with the ergodic invariant probability measure $\mu= (\log 2)^{-1}(x-y)^{-2}\,dA$ \cite{haas1, series, nakada2}.   

Define 
${\bf T}(x,\infty)=((1/x-[1/x],-[1/x])=(x_0,y_0)$. 
Let $\bT^{n}(x_0,y_0)=(x_n,y_n)$. By induction we have: if $x=[a_1,a_2,...]$  then 
\beqq
x_n=[a_{n+2},a_{n+3},...]\,  \text{and}\, y_n=-a_{n+1}-[a_n,...,a_1].\hskip .7in
\eeqq
Note that while $\bT$ will eventually run out of steam if $x$ is rational, there is no problem defining the map when $y$ is rational or $\infty.$

\subsection{Definition of the first return maps and their relation to thetas}

Set $\Omega_{\alpha}=\{(x,y)\in\Omega\,|\,(x-y)^{-1}<\alpha\}$. For $(x,y)\in\Omega_{\alpha}$, let $\tau_{\alpha}(x,y)=\min\{n\geq 1\,|\, {\bf T}^n(x,y)\in\Omega_{\alpha}$\}. Since ${\bf T}$ is ergodic, $\tau_{\alpha}$ is defined, with the exception of  a set of measure zero \cite{keller} (which we shall ignore). Now define the automorphism  ${\bf T}_{\alpha}$ on  $\Omega_{\alpha}$ by 
\beqq
{\bf T}_{\alpha}(x,y)={\bf T}^{\tau_{\alpha}(x,y)}(x,y).\hskip 1.5in
\eeqq
 We shall use the convention of denoting the $n^{th}$ iterate of a map with an exponent. ${\bf T}_{\alpha}$ is called the first return map on $\Omega_{\alpha}$.  Note that $\Omega_1=\Omega$ and $\bT_1=\bT.$

For $(u,v)\in\RR^2$ let $||(u,v)||=1/(u-v). $ 
Given $x\in(0,1)$ set $(\bx_k,\by_k)=\bT_{\alpha}^{k+1}(x,\infty)=\bT_{\alpha}^k(x_0,y_0).$ Aside from having some  inherent interest, the importance of $\bT_{\alpha}$ stems from the way in which it ties the sequence  $\bar{\Theta}(\alpha)$ to an ergodic dynamical system. 

\begin{lemma}\label{lemma1}

For almost each irrational $x\in (0,1)$, $\bar{\theta}_k=||\bT_{\alpha}^{k}(x_0, y_0)||$.
\end{lemma}

\begin{proof}

This  is an elaboration on $\theta_n=||\bT^{n}(x_0,y_0)||$ \cite{haasseries, haasmolnar}, which is proved in the following sequence of equalities.  
\beq\label{forkhin}
\frac{1}{x_n-y_n}=\frac{1}{[a_{n+2},...]+a_{n+1}+[a_n,...,a_1]}=\left(\frac{1}{T^n(x)}-\frac{q_{n-1}}{q_n}\right)^{-1}=\theta_n
\eeq
 where we refer to \cite{HW} and \cite{koksma}, respectively for the second and third equalities.

The proof is by induction. 
Let $m$ be the smallest value so that $\bT^m(x_0,y_0) =(x_m,y_m)\in \Omega_{\alpha}.$ It follows that  $\bT^m(x_0,y_0) =\bT_{\alpha}(x_0,y_0)$.  But this is also  equivalent to $m$ being the smallest value for which $\theta_m=\frac{1}{x_m-y_m}<\alpha$. Together these give, $\btheta_1=\theta_m=||\bT^{m}(x_0,y_0)||=||\bT_{\alpha}(x_0,y_0)||.$

Now to the inductive step. We suppose that $\bar{\theta}_k=||\bT_{\alpha}^{k}(x_0, y_0)||$. In terms of the natural extension this is  $\theta_{n_k}=||\bT^{n_k}(x_0,y_0)||$. 
Let $\theta_m=\theta_{n_{k+1}}=\btheta_{k+1}$. Then $||\bT^{m}(x_0,y_0)||=\theta_m<\alpha$ and for $n_k<r< m$, $ ||\bT^{r}(x_0,y_0)|| >\alpha.$ Translating  this says that $\bT^m(x_0,y_0)\in\Omega_{\alpha}$ and $\bT^r(x_0,y_0)\not\in\Omega_{\alpha}$ for $n_k<r< m$. Thus, $\bT^m(x_0,y_0) =\bT_{\alpha}^{k+1}(x_0, y_0).$ Putting this all together we have

\beqq
\btheta_{k+1}=\theta_{m}=||\bT^{m}(x_0,y_0)||=||\bT_{\alpha}^{k+1}(x_0, y_0)||.\hskip .5in
\eeqq

\end{proof}

By a theorem of Khinchin, for almost all $x\in (0,1) $  the limiting average of the partial quotients $a_i$ diverges to infinity \cite{khinchin}. Consequently, using the characterization of $\theta_n$ in \ref{forkhin}, for a.a $x$ the values $\theta_n$ get arbitrarily small.  It then follows from the lemma that for almost all $x$, $T^j(x_0,y_0)\in\Omega_{\alpha}$ for infinitely many positive integers $j$, and so the first return maps are all defined for points $(x,y)$ for $x$ in a set of positive measure in (0,1).

\subsection{Basic properties of the subsequences}

Our main tool in this section is the following theorem, which  will later be shown to hold when $\bT$ is replaced by one of the maps  $\bT_{\alpha}$ 

\begin{thm}\label{ally}\cite{bjw,jager,haasmolnar}
For $x$ in a full measure set $I_0\subset (0,1)$ and all $y\in (-\infty, -1)$, the sequence of points $\{{\bf T}^n(x,y)\}$ is distributed in $\Omega$ according to the density function 
 $f(x,y)=(\log 2)^{-1}(x-y)^{-2}$.
    \end{thm}
    
 We shall further suppose that for $x\in I_0$ the first return maps are all defined at $(x_0,y_0)$. 

Observe that $f$ is the density function for the $\bT$-invariant probability measure defined earlier, but the Ergodic Theorem is not sufficient in itself to guarantee convergence for the particular values appearing in the theorem, \cite{bjw, knuth}.
 Using   Theorem \ref{ally} we can easily prove the following version of the well known theorem of L\'evy \cite{levy}. It is interesting to  see that the placement of the log term in the constant is a consequence of  the same phenomena observed in D-L, where good approximations that are only first mediants are not accounted for by the continued fraction expansion when $\alpha>1/2$, \cite{bosma,HW}.

\begin{prop}
Given $\alpha\in(0,1]$,  for almost all $x\in (0,1)$  
\vskip .1in
 \hskip .4in$\d{\lim_{n\rightarrow\infty}\frac{\log\bar{q}_n}{n}= \left\{ \begin{array}{ll} 
    \pi^2(12(1-\alpha+\log 2 +\log \alpha))^{-1}& {\rm if} \,  \alpha>\frac{1}{2}\\
    \\
   \pi^2(12\alpha)^{-1} &  {\rm if}\,   \alpha\leq\frac{1}{2} .
\end{array} \right.
}$ 
 
  \end{prop}
\begin{proof}
 
Given $x\in I_0$,  $ n_k$ is the smallest value $n$ such that $k=  \#\{j\leq n \,|\, \theta_j<\alpha\}.$  By the lemma, $k$ is precisely $\#\{j\leq n \,|\, \bT^j(x_0,y_0)\in\Omega_{\alpha}\} $.  Thus , making use of Theorem \ref{ally},
\beq\label{important}
\lim _{k\rightarrow\infty}\frac{k}{n_k}=\lim _{k\rightarrow\infty}\frac{1}{ n_k}  \#\{j\leq n_k\,|\, \bT^j(x_0,y_0)\in\Omega_{\alpha}\} \hskip .8in
 \eeq
 
 \beqq
=  \lim _{n\rightarrow\infty}\frac{1}{ n }  \#\{j\leq n \,|\, \bT^j(x_0,y_0)\in\Omega_{\alpha}\}=\mu(\Omega_{\alpha}).\hskip 1.3in
\eeqq
This last value is $( \log 2)^{-1}(1-\alpha+\log 2 +\log \alpha)$ if $\alpha>1/2$ and is $( \log 2)^{-1} \alpha $ if $\alpha \leq 1/2$. 

Using the above and  the theorem of Levy, we have 
\beqq
\frac{\pi^2}{12\log 2 }=\lim_{n\rightarrow\infty}\frac{\log q_n}{n}=\lim_{k\rightarrow\infty}\frac{\log q_{n_k}}{n_k}=\lim_{k\rightarrow\infty}\frac{\log \bar{q}_{k}}{n_k}\hskip.7in
\eeqq
\beqq
=(\lim_{k\rightarrow\infty}\frac{k}{n_k})(\lim_{k\rightarrow\infty}\frac{\log \bar{q}_{k}}{k})=\mu(\Omega_{\alpha})\lim_{k\rightarrow\infty}\frac{\log \bar{q}_{k}}{k} \hskip 1.4in
\eeqq
Which proves the proposition.
\end{proof}

Henceforth we shall write $c_{\alpha}=(\log 2\,\mu(\Omega_{\alpha}))^{-1}$.
By a similar approach, one gets the following version of D-L.
\begin{prop}
 For $\alpha\in(0,1]$ and for almost all $x\in (0,1)$    the sequence $ \bar{\Theta}(\alpha)$ is distributed in the unit interval according to the density function \\
$c_{\alpha}(2\log 2 )^{-1}\zeta^{-1}(1-|1-2\zeta|).$ 
 \end{prop}


\subsection{Ergodic theory of the first return maps}  
 
The $\bT$ invariant  measure $\mu$ restricts to an invariant measure for ${\bf T}_{\alpha}$   on $\Omega_{\alpha}$, with respect to which $\bT_{\alpha}$ is ergodic \cite{keller}.  We normalize to get the invariant probability measure $\mu_{\alpha}=c_{\alpha}(x-y)^{-2}\,dA$. 
Of particular importance is the following result, the proof of which depends on Theorem \ref{ally} and the same trick used above.

\begin{thm}\label{ally2}
For almost all $x\in (0,1)$ and $y\in (-\infty, -1)$, the sequence of points $\{{\bf T}^n_{\alpha}(x,y)\}$  is distributed in $\Omega_{\alpha}$ according to the density function $ f_{\alpha}(x,y)=c_{\alpha}(x-y)^{-2} $.
   \end{thm}

\begin{proof}

Let $\bB$ be a Borel set in $\Oa$ with boundary of zero measure and suppose the limit
\beq\label{lim2}
\lim_{k\rightarrow\infty}\frac{1}{k} \#\{0<j<k\,|\, \bT_{\alpha}^j(x,y)\in \bB \}.\hskip 1.7in
\eeq
 exits and is equal to $\mu_{\alpha}(\bB)$. 
As a consequence of the Ergodic Theorem this holds for almost all $(x,y)\in\Omega_{\alpha}$.
Using formula (\ref{important}), the limit  (\ref{lim2}) can be rewritten
\beq\label{lim1}
\left(\lim_{k\rightarrow\infty}\frac{n_k}{k}\right )\, \left (\lim_{k\rightarrow\infty}\frac{1}{n_k} \#\{0<j<n_k\,|\, \bT^i(x,y)\in \bB \}\right )=\hskip .8in 
\eeq
\beqq
(c_{\alpha}\log 2)\,\lim_{n\rightarrow\infty}\frac{1}{n} \#\{0<j<n\,|\, \bT^j(x,y)\in \bB \}=(c_{\alpha}\log 2)\,\mu(B)=\mu_{\alpha}(B)
\eeqq
By Theorem \ref{ally}, if this  holds for $(x,y)$ then it will hold for $(x,y')$ for any $y'\in(-\infty,-1)$. 
Thus the limit (\ref{lim2}) also hold for $(x,y')$ for any $y'\in(-\infty,-1)$. Modulo some simple measure theoretic considerations, this implies the theorem.

\end{proof}

\section{The first return maps and the distribution of theta pairs}

In this section we turn to the space of pairs of the form $\{(\bar{\theta}_n,\bar{\theta}_{n+1})\}$and see how Jager's approach can be modified to derive the distribution function for the generic sequence of pairs.  
\subsection{Theta pairs for $\alpha\geq 1/2$}
The natural domain for the pairs, when $\alpha$ is taken to be greater than or equal to $1/2$,  is the set
\beqq
\Lambda_{\alpha}=\{(w,z)\in\RR^2\,|\, 0<w<\alpha,\, 0<z<\alpha,\, w+z<1\}\, 
\eeqq
Define
$
 \Lambda_{\alpha}^-=\{(w,z)\in \Lambda_{\alpha} \,|\, z<w-\alpha+\sqrt{1-4\alpha w}\}$ and $\Lambda_{\alpha}^+=\Lambda_{\alpha}\setminus \Lambda_{\alpha}^-.
$
 

On $\Lambda_{\alpha}$ we have the density function
\beqq
\lambda_{\alpha}(w,z)=
 \left\{ \begin{array}{ll} 
 c_{\alpha}(\sqrt{1-4\alpha zw})^{-1} & {\rm if} \, (w,z)\in \Lambda_{\alpha}^+ \\
 \\
  c_{\alpha}\big ( (\sqrt{1-4\alpha zw} )^{-1}+  ( \sqrt{1+4\alpha zw} )\big)^{-1}  &  {\rm if}\,  (w,z)\in \Lambda_{\alpha}^-.
\end{array} \right.
\eeqq

Jager's description of the distribution of approximating pairs becomes 
 \begin{thm}\label{talpha}
 For $\alpha\geq 1/2$ and almost all $x\in (0,1)$, the sequence $\{\bar{\theta}_k, \bar{\theta}_{k+1}\}$ is distributed in the region $\Lambda_{\alpha}$
 according to the density function $\lambda_{\alpha}(w,z)$. 
  In other words, for almost all $x\in (0,1)$ and for any Borel subset $B$ of $\Lambda(\alpha)$ with boundary of measure zero
 \beqq
 \lim_{n\rightarrow\infty}\frac{1}{n}\#\{j\leq n\, |\, (\bar{\theta}_j,\bar{\theta}_{j+1})\in B\}=\int_B d\lambda_{\alpha}=\lambda_{\alpha}(B).
 \eeqq
 \end{thm}

\subsection{ The structure of the first return maps for $\alpha\geq 1/2$ }
Supposing $\alpha\geq 1/2$, define the following sets:
 \beqq
 \Oa^-=\{(x,y)\in\Omega\,|\, y\leq\frac{\alpha x}{\alpha-x}\},\,\hskip 2.2in
 \eeqq
 \beqq
 \nabla_{\alpha}=\{(x,y)\in\Oa\,|\, y\geq x-\frac{1}{\alpha}\}= \Omega\setminus \Oa\,\text{and}\,
 \Oa^+=\Omega\setminus\left ( \Oa^-\cup \nabla_{\alpha}\right ).
 \eeqq 
 

  As usual we denote the closure of a set with an overline. Let $\nabla^*_{\alpha}$ denote the union of $\nabla_{\alpha}$ and its boundary along the curves $x=0$ and $x=y+1/\alpha.$

 \begin{lemma}\label{characterizeT1}
 When $\alpha\geq 1/2$, $\bT$ is a bijection, mapping  $\overline{\Omega_{\alpha}^-}$ onto  $ \nabla^*_{\alpha} $. Consequently, there is a simple dicotomy describing $\bT_{\alpha}:$
 \beqq
\bT_{\alpha}(x,y)=
 \left\{ \begin{array}{ll} 
\bT(x,y)& {\rm if} \, (x,y)\in \Omega_{\alpha}^+ \\
\\
\bT^2(x,y)& {\rm if} \, (x,y)\in \Omega_{\alpha}^- .
\end{array} \right.\hskip 1in
\eeqq
\end{lemma} 
 
 \begin{proof}
 
 We consider the action of $\bT^{-1}$ on $ \nabla_{\alpha}$. Since $1/2\leq \alpha\leq 1$, $[y]+1=-1.$ 
 It follows that on   $\nabla_{\alpha}$
 \beqq
(u,v)= \bT^{-1}(x,y)=\bigg ( \frac{1}{x-[y]-1},  \frac{1}{x-[y]-1}\bigg)\,=\, \bigg ( \frac{1}{x+1}, \frac{1}{y+1}\bigg ).
 \eeqq 
 Note that this extends to the diagonal $x=y+1/\alpha$, as well as the line $x=0$
 
 On the diagonal we have 
 \beqq
 (u,v)= \bT^{-1}(x,y) = \bigg ( \frac{1}{y+1+\frac{1}{\alpha}}, \frac{1}{y+1}\bigg ).\hskip .8in 
 \eeqq
$u$ and $v$ are thus related by the equation $\frac{{\alpha} u}{\alpha-u}=v$, which is the curve $\gamma$ bounding $\Omega_{\alpha}^- .$ 
 
 Setting $x=0$ gives the line $u=1$. As $y$ approaches -1, $v$ goes to infinity along the vertical line with  $u=1/(x+1)$. In fact the whole strip $-2\leq y< -1$ maps to the strip $1/2\leq x  <1$. It follows that $\bT^{-1}$ is a bijection of $\nabla_{\alpha}$ onto $\Omega_{\alpha}^- $ which extends to the above line segments. Therefore the inverse is a bijection as stated, extending to the boundary of   $\Omega_{\alpha}^- $.
 
 The final assertion of the lemma follows easily. If $(x,y)\in \Omega_{\alpha}^+ $ then $\bT(x,y)\not\in \nabla_{\alpha}$. Therefore $\bT(x,y) \in \Oa$, as asserted. 
 On the other hand if $(x,y)\in \Omega_{\alpha}^- $ then $\bT(x,y)\in \nabla_{\alpha}$; in particular,  $\bT(x,y)\not\in \Oa$. But since $\nabla_{\alpha}$ is disjoint from  $\Omega_{\alpha}^- $,  $\bT^2(x,y)\in \Oa$. 
 
  \end{proof} 
 
 \subsection{From $\Oa$ to the domain of theta pairs}
 
 Define
 \beqq
 F^+(x,y)=\frac{- x y }{ x -y }\,\,\,\text{and}\,\,\, 
 F^+(x,y) =\frac{(1- x )(1- y )}{ x -y }\hskip .4in
  \eeqq
 and then set
\beqq
\bF( x,y) = 
  \left\{ \begin{array}{ll} 
  \big(||( x,y)||,  F^+( x,y)\big)   & {\rm if} \,  ( x,y)\in \Omega_{\alpha}^+ \\
 \big(||( x,y)||,  F^-( x,y)\big)  & {\rm if} \,  ( x,y)\in\Omega_{\alpha}^- .
\end{array} \right.\hskip .5in
\eeqq

The next proposition elucidates the relationship between the space $\Oa$ and the natural domain for the theta pairs. In effect it allows us to equate the $\bT_{\alpha}$-orbit of a point in $\Oa$ with a sequence of pairs in $\Lambda_{\alpha}$. It is an easy step from here to the proof of Theorem \ref{talpha}.
\begin{prop}\label{prop}.\\
a) $\bF( x,y)=(w,z)$ maps $\Oa$ injectively onto   $\Lambda_{\alpha}$. Its inverse is given by
\beqq
{\bf H} ( w,z)=\bigg( \frac{1-\sqrt{1-4wz}}{2w},\frac{-1-\sqrt{1-4wz}}{2w}\bigg )\hskip 1.1in 
 \eeqq
 b)  $\bF( x,y)=(w,z)$ maps $\Oa^-$ injectively onto $\Lambda^-_{\alpha}$. Its inverse is given by
\beqq
{\bf H}^{-} (w,z)=\bigg( \frac{2w+1-\sqrt{1+4wz}}{2w},\frac{2w-1-\sqrt{1+4wz}}{2w}\bigg )\hskip .4in
\eeqq
c) Furthermore, for almost all $x\in (0,1) $ and $(x_0,y_0)=\bT(x,\infty)$, 
\beqq
\bF(\bar{x}_k,\bar{y}_k)=\bF\big(\bT_{\alpha}^k(x_0,y_0)\big)=(\btheta_k,\btheta_{k+1}).\hskip .3in
\eeqq

\end{prop}

\begin{proof}
The injectivity in the first statement is proved and argument for ${\bf H}^-$, which is much the same but  more laborious, is left to the reader. 

One sees easily that $\bF\circ{\bf H}(w,z)=(w,z).$ 
We  complete the argument by showing that $\bF$ is one-to-one.
If not then we can find  $( x,y)$ and $  ( u,v) $ so that $|| ( x,y) ||=|| ( u,v) ||$ and
$F^+( x,y)=F^+( u,v)$. The first becomes  
\beq\label{inverse1}
x-u=y-v,\hskip 2.2in
\eeq
and the second takes the form 
\beq\label{inverse2}
xy(u-v)= uv(x-y).\hskip 2in\eeq
Then \ref{inverse2} is shuffled to give $xu(y-v)= yv(x-u)$ and substituting in for $y-v$ from \ref{inverse1} and canceling,  we get $xy=uv$. Taking this back to \ref{inverse2}, we can cancel the terms in front on both sides and then substitute in for  $u= (xy)/v$. Multiplying through by $v$ and simplifying gives $v^2-v(x-y)-xy=0$. Since $v$ must be negative we get  $v  =y$ and the result follows.

You may have noticed that, in fact, both $F^+$ and $F^-$ are defined on all of $\Omega$ and that they have the inverses $ {\bf H}$ and ${\bf H}^-  $, which are defined on the triangle $w+z<1$ in the first quadrant.

The map $\bF$ is continuous and both $F^+$ and $F^-$ extend continuously to the boundaries of their domains. One then checks that the maps on the boundary are onto the boundaries of $\Lambda$ and $\Lambda^-$. This includes the "boundary at infinity"  mapping to the remaining boundary component. For the sake of honesty we check b). 

Consider the curve $y=\frac{\alpha x}{\alpha -x}.$ Its image under $F^-$ has  
\beq\label{wx}
 w=\frac{x-\alpha}{x^2}\,\,\text{and therefore}\,\, x=\frac{1-\sqrt{1-4w\alpha}}{2w}.\hskip .5in
\eeq
Now substituting in for $y$ and using the first and second parts of \ref{wx} we have
\beqq
z=\frac{(1-x)(1-\frac{\alpha x}{\alpha -x})}{x- \frac{\alpha x}{\alpha -x}}=
\frac{x-\alpha}{x^2}+\frac{2\alpha}{x}-1-\alpha\hskip .8in
\eeqq
\beqq
= w+\frac{2\alpha}{x}-1-\alpha=w-\alpha+\sqrt{1-4w\alpha}\hskip 1.1in
\eeqq
which is the curved boundary component $\gamma$ of $\Lambda^-$.

The image of the  piece of the boundary with  $x=1$ and $y\in (\frac{\alpha}{\alpha-1},-\infty)$, is the segment of $z=1$ with $w\in(0,1-\alpha)$. As $y\rightarrow-\infty$, $w$ goes to zero and $z$ takes on values between 0 and $1-\alpha$. 

That completes parts a) and b) of the lemma. We still need to address part c), which says that given $x\in (0,1)$, $\bF$ takes the sequence $(\bar{x}_k,\bar{y}_k)$ to the corresponding sequence of theta pairs.

Suppose we have  $x$ generic for the conclusion, $\bar{\theta}_k=||\bT_{\alpha}^{k}(x_0, y_0)||$, of Lemma \ref{lemma1}. Since $\bar{\theta}_k=||(\bar{x}_k,\bar{y}_k)||$, the result is clear for the first entry, so we turn to the second.

By Lemma \ref{characterizeT1}    there are only two possibilities   for the value of $\bar{\theta}_{k+1}$. First, if $(\bar{x}_k,\bar{y}_k)\in \Oa^+$, then $\bT_{\alpha}(\bar{x}_k,\bar{y}_k)=\bT(\bar{x}_k,\bar{y}_k)$ and
\beqq
\bar{\theta}_{k+1} =||\bT_{\alpha}(\bar{x}_k,\bar{y}_k)||=||\bT(\bar{x}_k,\bar{y}_k) ||=||\frac{1}{\bar{x}_k}-[ \frac{1}{\bar{x}_k}], \frac{1}{\bar{y}_k}-[ \frac{1}{\bar{x}_k}] || 
\eeqq
\beqq
=\frac{- \bar{x}_k \bar{y}_k}{ \bar{x}_k-\bar{y}_k} =  F^+(\bar{x}_k,\bar{y}_k)=  \bF(\bar{x}_k,\bar{y}_k).\hskip 1in
 \eeqq 
The second possibility occurs when  $(\bar{x}_k,\bar{y}_k)\in \Oa^-$. Then  we  must have $1/2\leq \alpha < \bar{x}_k <1$ and consequently $[1/ \bar{x}_k ]=1$. Using this fact and a little calculation gives
 
\beqq
\bar{\theta}_{k+1} =||\bT_{\alpha}(\bar{x}_k,\bar{y}_k)||=||\bT^2(\bar{x}_k,\bar{y}_k)|| = \bigg(\frac{\bar{x}_k \bar{y}_k}{(1-x[\frac{1}{x}])(1-y[\frac{1}{x}])}\bigg )^{-1}
\eeqq
\beqq
=\frac{(1- \bar{x}_k)(1- \bar{y}_k)}{ \bar{x}_k \bar{y}_k}=  F^-(\bar{x}_k,\bar{y}_k)=  \bF(\bar{x}_k,\bar{y}_k) .\hskip .4in
\eeqq    

That completes the proof of the proposition.

\end{proof}

\subsection{The distribution}
The following is a strengthened version of Theorem \ref{talpha}.  

\begin{thm}\label{dist}
For almost all $(x_0,y_0)\in \Oa$ the sequence $(w_k,z_k)=\bF (\bar{x}_k,\bar{y}_k)=\bF(\bT_{\alpha}^k(x_0,y_0))$ is distributed in the region $\Lambda_{\alpha}$ according to the density function $\lambda_{\alpha}$. Furthermore, this holds for almost all $x\in (0,1)$ with $(x_0,y_0)=\bT(x,\infty)$ and in this case $(w_k,z_k)=(\btheta_k,\btheta_{k+1}).$

\end{thm}

\begin{proof}
This is a modification of the proof in \cite{jager} along the lines of \cite{haasmolnar2}.
Abusing notation we define the measure $\lambda_{\alpha}$ on $\Lambda_{\alpha}$ by setting $\lambda_{\alpha}(D)=\mu_{\alpha}(\bF^{-1}(D))$ for an open set $D$. Let $D^-=D\cap\Lambda^-$. Then following a messy computation of Jacobians, we have \\
\beqq
\lambda_{\alpha}(D)=\int\int_{F^{-1}(D)}f_{\alpha}\,dxdy =\int\int_Df_{\alpha}({\bf H}(w,z))\,|\text{Jac}{\bf H}(w,z)|\,dwdz 
\eeqq
\beqq
+\int\int_{D^-}f_{\alpha}({\bf H}^-(w,z))\,|\text{Jac}{\bf H}^-(w,z)|\,dwdz = \int\int_D\lambda_{\alpha}(w,z)\,dwdz,
\eeqq
thus justifying our abusive notation.

Since $\bT_{\alpha}$ is ergodic with invariant probability measure $\mu_{\alpha}$, for almost all $(x_0,y_0)\in\Oa$
\beq\label{above}
\lambda_{\alpha}(D)=\mu_{\alpha}(\bF^{-1}(D))=\lim_{k\rightarrow\infty}\frac{1}{k}\#\{j\leq k\,|\,   (\bar{x}_k,\bar{y}_k)\in \bF^{-1}(D)\}
\eeq
\beqq
=\lim_{k\rightarrow\infty}\frac{1}{k}\#\{j\leq k\,|\,   (w_k,z_k)\in  D \}.\hskip 1in
\eeqq
That proves the first assertion of the theorem.

Now suppose $x\in (0,1)$ is chosen from the full measure set $I_0$, guaranteed by Theorem \ref{ally2}, for which the sequence  $(\bar{x}_k,\bar{y}_k)$ is distributed according to the density function $f_{\alpha}$. Then as a consequence of Proposition \ref{prop} and \ref{above} above we get
\beqq
 \lim_{n\rightarrow\infty}\frac{1}{n}\#\{j\leq n\, |\, (\bar{\theta}_j,\bar{\theta}_{j+1})\in D\}
=\lim_{k\rightarrow\infty}\frac{1}{k}\#\{j\leq k\,|\,   (w_k,z_k)\in  D \}=\lambda_{\alpha}(D).
\eeqq
\end{proof}

\begin{figure}[here] 
  \centering
  \subfloat[$\alpha=1$]{\label{fig:gull}\includegraphics[width=0.33\textwidth]{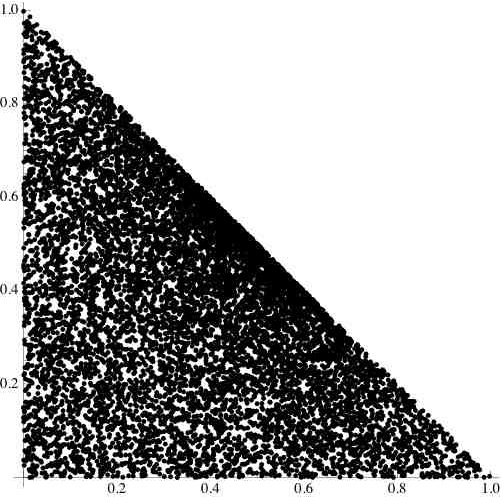}}                
  \subfloat[$\alpha=.7$]{\label{fig:tiger}\includegraphics[width=0.33\textwidth]{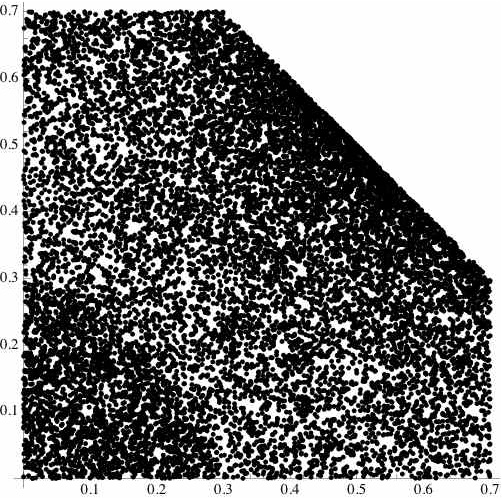}}
   \subfloat[$\alpha=.5$]{\label{fig:mouse}\includegraphics[width=0.33\textwidth]{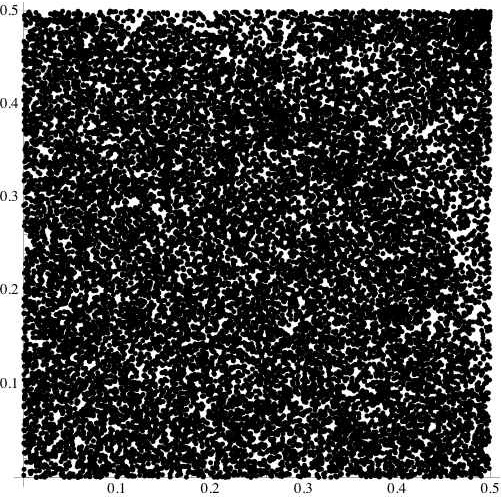}}
  \caption{ }
  \label{animals}
\end{figure}


 
In the three graphs of Figure \ref{animals},  the sets of pairs $(\bar{\theta}_j,\bar{\theta}_{j+1})$ are plotted for different values of $\alpha$. The $\bar{\theta}$'s have been extracted from a sequence of $\theta$'s, generated by taking $x=\pi^2+\sqrt{2}-1$.   When $\alpha=1$ you see the distribution described by Jager's theorem \cite{jager}. For $\alpha<1/2$ it is difficult to see any interesting detail using this approach.

\section{A geometric approach that addresses all values $\alpha$}
  In this final section we shall revisit the first return maps using a more geometric approach. 
  
   \subsection{The setup}
 
  Given a fraction $p/q$ in lowest terms and a number $\kappa\in (0,1]$,  $\bD_{\frac{p}{q}}(\kappa)$ is the disc of radius $\kappa/q^2$ which is tangent to the real line at the point $p/q$. Set $\bD_{\infty}(\kappa)=\{\zeta=u+iv\,|\,v>2/\kappa\}$. These regions are sometimes called horocycles or  their boundaries are called Ford circles. Let $\mathcal{D}(\kappa)$ denote the union of all such discs. We will be particularly interested when $p/q$ is in the closed unit interval.
 
 Suppose $x$ is irrational with continued fraction expansion $x=[a_1,a_2,...]$. As in \cite{sinai} define 
\beqq
\Delta=\Delta_{a_1,...a_{n+1}}^{n+1}=\{\zeta=[a_1,..,a_{n+1}+r]\,|\, r \in (0,1),\text{irrational}\}. 
\eeqq
This interval has endpoints
\beqq
\frac{p_{n+1}}{q_{n+1}}=[a_1,...,a_{n+1}]\,\,\,\text{and}\,\,\, \frac{p_{n+1}+p_n}{q_{n+1}+q_n}=[a_1,...,a_{n+1}+1].
\eeqq

  $T^{n+1}$ restricted to  $\Delta$ closure is a M\"obius transformation $g$ mapping $\bar{\Delta}$,  1-1 and onto  [0,1] and taking $\frac{p_n}{q_n}$ to $\infty$. It is easy to  verify that 
  \beq\label{matrix}
  g_{n+1}(x)=\pmtx{q_{n+1}&-p_{n+1} \cr -q_n&p_n}(x)=\frac{q_{n+1} x -p_{n+1} }{-q_n x +p_n}.  
    \eeq 
Then for any $y<-1$, $\bT^{n+1}(x,y)=(g_{n+1}(x),g_{n+1}(y))$ \cite{haasmolnar}. 

The transformation $g_n$ is  an automorphism of the Riemann sphere. It will preserve the upper half-plane $\HH$ when $n$ is even and it will interchange the upper and lower half-planes when $n$ is odd. We shall extend $g_n$ to a self-map of $\HH$ by setting 
\beqq
G_n(z)=
 \left\{ \begin{array}{ll} 
g_n(z) & {\rm if} \,\, n\, \text{even} \\
\\
\overline{g_n(z)}  & {\rm if} \,\, n\, \text{odd} .
\end{array} \right.\hskip 1.5in
\eeqq

Given $x\not= y$, real numbers or infinity, let $\overline{xy}$ denote the arc of the circle in the upper half-plane $\HH$  orthogonal to $\RR\cup\{\infty\}=  \hat{\RR}$. This is a geodesic in the Poincar\'e  model for the hyperbolic plane. We'll have $\bT^{n+1}$ act on geodesics by setting $\bT^{n+1}(\overline{xy})=  \overline{g_{n+1}(x)g_{n+1}(y)}=G_{n+1}(\overline{xy}).$ 

One begins to see how this fits with the earlier material in the following.

\begin{prop}
 For $x$ irrational in $(0,1)$ with convergents $\p_j/q_j$, any $y<-1$ and $n\in\NN$
 \beqq
 ||\bT^{n+1}(x,y)||<\alpha\,\,\,\text{if and only if}\,\,\, \overline{xy}\cap\bD_{\frac{p_n}{q_n}}(\alpha)\not=\emptyset
 \eeqq 
 \end{prop}
 
 \begin{proof}
The result holds for $n=0$ by taking $\frac{p_0}{q_0}=\frac{0}{1}=0$.
 
  A simple calculation \cite{haas} shows that the transformations $G_n$ permute the discs in $\mathcal{D}(\alpha).$   
Then, since  $g_{n+1}(p_n/q_n)=\infty$, $G_{n+1}(\bD_{\frac{p_n}{q_n}}(\alpha))=\bD_{\infty}(\alpha)$. Therefore $\overline{xy}\cap\bD_{\frac{p_n}{q_n}}(\alpha)\not=\emptyset$ if and only if $\overline{\bT^{n+1}(x,y)}\cap \bD_{\infty}(\alpha)\not=\emptyset$. But this last is equivalent to $(x-y)/2>2/\alpha$ or $ ||\bT^{n+1}(x,y)||<\alpha$. 
  \end{proof}

 \subsection{ Geodesic-horocycle intersections and $\tau_{\alpha}$   }
 
Given $(x,y)\in \Oa$, let $\bD(x,y)=\bD_{\frac{p_n}{q_n}}(\alpha) $ be the horocycle intersecting $\overline{xy}$ with $q_n$ minimal. If $q_n=1$ and $\bD_0(\alpha)\cap\overline{xy}\not=0$ then  set $\bD(x,y)=\bD_0(\alpha),$ otherwise set $D(x,y)=\bD_1(\alpha).$

\begin{thm}\label{lastthm}
Suppose $\alpha\geq 1/2$. Then $\bD(x,y)$ is either $\bD_0(\alpha)$   or $\bD_1(\alpha)$ and
\beq\label{1/2}
 \tau_{\alpha}(x,y)=
 \left\{ \begin{array}{ll} 

1  & {\rm if} \,\, \bD(x,y)=\bD_0(\alpha) \\
2  & {\rm if} \,\, \bD(x,y)=\bD_1(\alpha)  .
\end{array} \right.\hskip 1.5in
\eeq
 
 Suppose $\alpha<1/2$. Then $\bD(x,y)$ is either $\bD_0(\alpha)$ or  $\bD_1(\alpha)$ and  the conclusion of \ref{1/2} holds or else $\bD(x,y)=\bD_{\frac{p_n}{q_n}}(\alpha)$ with $q_n>1$ and then
 \beqq
  \tau_{\alpha}(x,y)=
 \left\{ \begin{array}{ll} 

n+1  &  \text{\rm if {\it n} is odd and}\,\,  x<\frac{p_n}{q_n}\,\, \text{\rm or if {\it n} is even and} \,\, x>\frac{p_n}{q_n} \\
n+2  &  \text{\rm if {\it n} is even and}\,\,  x<\frac{p_n}{q_n}\,\, \text{\rm or if {\it n} is odd and} \,\, x>\frac{p_n}{q_n}  .
\end{array} \right.\hskip 1.5in
\eeqq
 
\end{thm} 

One could use this point of view to characterize $\Oa^+$ and $\Oa^-$ and reprove Lemma \ref{characterizeT1}. In Example 1, we will see how one might address the next simplest case with $\alpha<1/2.$ 
 
 \begin{proof}
  First observe that if $\alpha\geq 1/2$, every $\overline{xy}$ must meet one of   $\bD_0(\alpha)$ or $\bD_1(\alpha)$. 
Suppose $D(x,y)=\bD_0(\alpha)$. Then  $\overline{xy}\,\cap\bD_0(\alpha)\not=\emptyset,$ and by  the proposition, $ ||\bT (x,y)||<\alpha.$  In other words, $\tau_{\alpha}(x,y)=1$.     If $\overline{xy}\cap\bD_0(\alpha) =\emptyset,$ then $\overline{xy}\cap\bD_1(\alpha) \not=\emptyset.$  Then $x>1/2$ and consequently $\frac{p_1}{q_1} = \frac{1}{1}.$  Again, it follows from  the proposition that $\tau_{\alpha}(x,y)=2$. Note that \ref{1/2} remains true even when $\alpha\leq 1/2$.

Henceforth we take $\alpha\leq 1/2$. Suppose $\bD(x,y)=\bD_{ \frac{1}{a_1}} (\alpha)$ for $a_1>1$. If $x>\frac{1}{a_1}$ then an easy computation shows that $x<\frac{1}{a_1-1}$ or, in other words, $x\in \Delta^1_{a_1-1}.$ Let $\bT(x,y)=(x',y').$ $\bT$ maps the closure of  $\Delta^1_{a_1-1}$ to [0,1], taking $\frac{1}{a_1}$ to 1 and the geodesic $\overline{xy}$ to $\overline{x'y'}$, which intersects $\bD_1$. It follows from the previous paragraph that $\tau_{\alpha}(x',y')=2$ and consequently that $\tau_{\alpha}(x,y)=3$. Similarly, if $x<\frac{1}{a_1}$ then $x>\frac{1}{a_1+1}$ or $x\in \Delta^1_{a_1}.$  $\bT$ maps the closure of  $\Delta^1_{a_1}$ to [0,1], taking $\frac{1}{a_1}$ to 0 and the geodesic $\overline{xy}$ to $\overline{x'y'}$, which intersects $\bD_0(\alpha)$. Then $\tau_{\alpha}(x',y')=1$ and $\tau_{\alpha}(x,y)=2$. The theorem follows for $n=1$ where $\frac{p_1}{q_1}=\frac{1}{a_1}.$

The proof is completed by induction. Suppose the result holds for $k>1$. $\frac{p_{k+1}}{q_{k+1}}$ is contained in one of the intervals  
$ \Delta^1_{a_m} $ Thus, on the interval, and therefore in a neighborhood of the fraction, $\bT$ is the M\"obius transformation $g(z)=(1/z)-a_m.$ Then, as above, $\bT$ maps the geodesic $\overline{xy}$ to $\overline{x'y'}.$ $g$ reverses orientation. Consequently, if $x$ is greater than or less than $\frac{p_{k+1}}{q_{k+1}}$ then the reverse is true for $\bT(x)$ relative to $\frac{p_{k+1}}{q_{k+1}}$. In particular, if $x>\frac{p_{k+1}}{q_{k+1}}$,  then $T(x)=x'<\frac{p_{k}}{q_{k}}$. Then  $\bD(x',y')= \bD_{ \frac{p_{k}}{q_{k}} }(\alpha)$ and by the inductive hypothesis, $\tau_{\alpha}(x',y')$ is $k+1$ if $k$ is odd and it is $k+2$ if $k$ is even. It follows that $\tau_{\alpha}(x,y)$ is $k+2$ if $k+1$ is even and it is $k+3$ if $k+1$ is odd. In the same manner the result follows when $x<\frac{p_{k+1}}{q_{k+1}}$. That completes the proof.
\end{proof}

\noindent {\bf Example 1.}
  It is possible to choose $\epsilon_0 $ small enough so that for $\alpha=1/2-\epsilon_0$, $\overline{xy}$ must intersect one of the horocycles $\bD_0(\alpha), \bD_1(\alpha)$ or $\bD_{\frac{1}{2}}(\alpha) $ for any $(x,y)\in\Oa$.    If the first or the second hold then $\tau_{\alpha}(x,y)$ is respectively 1 or 2. In the remaining case  $\overline{xy}$ meets  $\bD_{\frac{1}{2}}(\alpha) $ but neither of the other horocycles. If $x<1/2$ then by the theorem  $\tau_{\alpha}(x,y)=2$ whereas, if  $x>1/2$, then  $\tau_{\alpha}(x,y)=3.$ It is possible, although tedious, to numerically characterize the regions on which each of the behaviors holds. Then $\Oa$ will be divided into three regions and by specifying the appropriate power $\tau_{\alpha}(x,y)$  of $\bT$ on each of them, one can describe  $\bT_{\alpha}$. One could then, in principle,   compute the density function for the distribution of theta pairs as in Theorem \ref{dist}.
  
The dynamical systems determined by values $\alpha$ with $1/2-\epsilon_0\leq  \alpha<1/2$ will be structurally identical, where the sets $\Oa$ are subdivided into  three regions of the same shape, on which $\tau_{\alpha}(x,y)$ is constant. We conjecture that there will be an infinite, discrete set of numbers $\alpha_i$ decreasing to some value $\alpha_{\infty}$ greater than zero, so the for $\alpha_{i+1}<\alpha<\alpha_{i} $ the dynamical systems $\{\bT_{\alpha}, \Oa\}$ are topologically conjugate. At one of the values $\alpha_i$, the system will abruptly change, in particular $\sup\tau_{\alpha}(x,y)$ will increase.  Below $\alpha_{\infty}$ the the systems should attain a higher degree of complexity. All this should somehow relate to the Markoff spectrum \cite{cf}.
 \vskip .2in
 
 \noindent {\bf Example 2.}
 It follows in the previous example that if $\alpha>1/2-\epsilon_0$ then $\tau_{\alpha}(x,y)$ is defined for all $(x,y)\in\Oa.$ This will not hold in general for all values of $\alpha$. In fact the picture is quite complex. In the simplest case, consider what happens when we take $\zeta=\frac{1}{2}(\sqrt{5}-1).$ By an old theorem of Hurwitz \cite{HW}, there are only finitely many fractions $p/q$ with
 \beqq
 |\zeta-\frac{p}{q}|<\frac{1}{\sqrt{5}q^2}.\hskip 1.4in
 \eeqq
 
This inequality can be read geometrically as saying that $\overline{\zeta\infty}$ meets only finitely many of the horocycles  $\bD_{\frac{p_n}{q_n}}(\frac{1}{\sqrt{5}}),$ where the fractions are the continued fraction convergents of $\zeta.$ 
The convergent $\frac{p_1}{q_1}=1$ is one of these, so $\overline{\zeta\infty}\cap D_1\not=\emptyset.$ Then $Y$ can be chosen so that for $y<Y$, 
$\overline{\zeta y}\cap D_1(\frac{1}{\sqrt{5}})\not=\emptyset$, the arcs are disjoint from $D_0(\frac{1}{\sqrt{5}})$ and $(\zeta,y)\in\Omega_{\frac{1}{\sqrt{5}}}$. By Theorem \ref{lastthm}, $\tau_{\frac{1}{\sqrt{5}}}(\zeta,y)=2.$   But since $\overline{\zeta y}$ is asymptotic to $\overline{\zeta\infty}$, by Hurwitz's Theorem it can only meet finitely many of the horocycles $\bD_{\frac{p_n}{q_n}}(\frac{1}{\sqrt{5}}).$ Consequently, for some $N$ we'll have $(x_N,y_N)\in\Omega_{\frac{1}{\sqrt{5}}}$ but $\bT(x_n,y_n)\not \in\Omega_{\frac{1}{\sqrt{5}}} $ for any $n>N$. Thus $\tau_{\frac{1}{\sqrt{5}}}(x_N,y_N)$ is not defined. 

As in Example 1, as $\alpha$ decreases one expects that the sets on which $\tau_{\alpha}$ is {\em not} defined will increase in complexity. We conjecture that there exists a value $\alpha'_{\infty}$ so that if $\alpha >\alpha'_{\infty}$, the set $S_{\alpha}=\{x\,|\, \tau_{\alpha}\,\text{is not defined}\}$ is finite whereas for $\alpha <\alpha'_{\infty}$ the complexity and probably the Hausdorff dimension will increase with decreasing $\alpha$.

\vskip .2in


\end{document}